\documentclass[11pt,reqno]{amsart}
\usepackage{amsmath,amsfonts,amsthm,amssymb,color,hyperref}

\usepackage[T1]{fontenc}

\usepackage{graphicx}

\usepackage{subfigure} 

\makeatletter
     \def\section{\@startsection{section}{1}%
     \z@{.7\linespacing\@plus\linespacing}{.5\linespacing}%
     {\bfseries
     \centering
     }}
     \def\@secnumfont{\bfseries}
     \makeatother

\usepackage{graphicx}

\newcommand{\E}{\mathbb E}
\newcommand{\pr}{\mathbb P}

\newcommand{\1}{{\bf 1}}

\newtheorem{theorem}{Theorem}[section]
\newtheorem{lemma}[theorem]{Lemma}
\newtheorem{proposition}[theorem]{Proposition}
\newtheorem{corollary}[theorem]{Corollary}
\theoremstyle{definition}

\theoremstyle{remark}
\newtheorem{remark}{Remark}
\newtheorem{example}{Example}
\numberwithin{equation}{section}
\begin{document}
\title[Sharp rate]{On sharp rate of convergence for discretisation of integrals driven by fractional Brownian motions and related processes with discontinuous integrands}
\author[E. Azmoodeh]{Ehsan Azmoodeh}
\address{University of Liverpool, Department of Mathematical Sciences, Liverpool, UK}
 \email{ehsan.azmoodeh@liverpool.ac.uk}
\author[P. Ilmonen]{Pauliina Ilmonen}
\address{Aalto University School of Science, Department of Mathematics and Systems Analysis, Espoo, Finland}
 \email{pauliina.ilmonen@aalto.fi}
 \author[N. Shafik]{Nourhan Shafik}
\address{Aalto University School of Science, Department of Mathematics and Systems Analysis, Espoo, Finland}
 \email{nourhan.shafik@aalto.fi}
 \author[T. Sottinen]{Tommi Sottinen}
\address{University of Vaasa, School of Technology and Innovations, Vaasa, Finland}
\email{tommi.sottinen@uwasa.fi}
\author[L. Viitasaari]{Lauri Viitasaari}
\address{Uppsala University, Department of Mathematics, Uppsala, Sweden}
\email{lauri.viitasaari@math.uu.se}

\begin{abstract}
We consider equidistant approximations of stochastic integrals driven by H\"older continuous Gaussian processes of order $H>\frac12$ with discontinuous integrands involving bounded variation functions. We give exact rate of convergence in the $L^1$-distance and provide examples with different drivers. It turns out that the exact rate of convergence is proportional to $n^{1-2H}$ that is twice better compared to the best known results in the case of discontinuous integrands, and corresponds to the known rate in the case of smooth integrands. The novelty of our approach is that, instead of using multiplicative estimates for the integrals involved, we apply change of variables formula together with some facts on convex functions allowing us to compute expectations explicitly. 
\end{abstract}

\maketitle



\allowdisplaybreaks


\section{Introduction}
We consider the rate of convergence for equidistant approximations of pathwise stochastic integrals 
\begin{equation}\label{eq:int0}
\int_0^1 \Psi'(X_s)\, dX_s \approx \sum_{k=1}^n \Psi'(X_{t_{k-1}})(X_{t_k}-X_{t_{k-1}}),
\end{equation}
where $t_k = \frac{k}{n}$.
Here $\Psi$ is a difference of convex functions and $X$ is a centered Gaussian process with non-decreasing variance function $V(s) = \E X_s^2$ normalized such that $V(1)=1$.  We assume that the variogram function 
$$
\vartheta(t,s) = \E(X_t-X_s)^2
$$
satisfies, for some $H\in(\frac12,1)$, that 
\begin{equation}
\label{eq:variogram-assumption}
\vartheta(t,s) = \sigma^2|t-s|^{2H} + g(t,s),
\end{equation}
where
$$
\lim_{|t-s|\to 0}\frac{g(t,s)}{|t-s|^{2H}} = 0.
$$
This means, in particular, that the process $X$ has $H$ as its H\"older index. One way to realize the process $X$ is to take fractional Brownian motion $B^H$, with index $H$ and an independent Gaussian process $G$ with variogram $g$ (such process has H\"older index at least $H$) and put
$$
X_t = X_0 + B_t^H + G_t,
$$
where $X_0$ may be random initial (Gaussian) value.
We also note that we have either $V(0)=C>0$ (e.g. stationary case) or $V(s)\ge c s^{2H}$ (e.g. the case of the fractional Brownian motion).  It follows that 
$$
\int_0^1 \frac{1}{\sqrt{V(s)}} \,ds < \infty.
$$
Consequently, by \cite{Chen-Leskela-Viitasaari-2019} the pathwise Riemann-Stieltjes stochastic integral in \eqref{eq:int0} exists and we have the classical chain rule
\begin{equation}
\label{eq:ito}
\Psi(X_1)-\Psi(X_0) = \int_0^1 \Psi'(X_s)\,dX_s.
\end{equation}

In the case of the fractional Brownian motion, the problem was studied in \cite{Azmoodeh-Viitasaari-2015b}. This article extends the article \cite{Azmoodeh-Viitasaari-2015b} into two directions: (i) we allow more integrators than just the fractional Brownian motion and (ii) we give exact $L^1$ error of the approximations. Rather surprisingly, it turns out that we obtain the rate $n^{1-2H}$ that is twice better compared to the rate obtained in \cite{Azmoodeh-Viitasaari-2015b} and corresponds to the known correct rate in the case of smooth functions $\Psi'$ (see for instance \cite{Azmoodeh-Viitasaari-2015b,Garino-et-al} and the references therein). In contrast in the Brownian motion case, introducing jumps reduces the rate into $n^{-1/4}$ in comparison to $n^{-1/2}$ obtained for smooth functions $\Psi'$ (see, e.g. \cite{Azmoodeh-Viitasaari-2015b}). For other related articles on stochastic integrals with discontinuous integrands, see also \cite{Chen-Leskela-Viitasaari-2019,Hinz-Tolle-Viitasaari-2020,Hinz-Tolle-Viitasaari-2022,Sottinen-Viitasaari-2016a,Yaskov-2019}. 

The rest of the article is organized as follows: the main results are give in Section \ref{sect:main}.  In Section \ref{sect:exa} we give examples. Finally, the proofs are given in Section \ref{sect:proofs}. 

\section{Statement of the main results}\label{sect:main}
We begin by recalling some basic facts on convex functions and on functions of bounded variation. For details on the topic, see for instance \cite{Revuz-Yor}. 

For a convex function $\Psi$, let $\Psi'$ denote its one sided derivative. Then the derivative $\Psi'' = \mu$ exists as a Radon measure. A particular example includes the function $\Psi(x) = |x-a|$, in which case $\Psi'(x) = sgn(x-a)$ and $\Psi''(x) = \delta_a(x)$, the Dirac measure at level $a$. More generally, if $\Psi'$ is of (locally) bounded variation, then it can be represented as the difference of two non-decreasing functions. As a corollary, $\Psi'$ can be regarded as the derivative of a function $\Psi$ that is a difference of two convex functions. That is, we have $\Psi = \Psi_1-\Psi_2$ and the second derivative $\Psi''$ is a signed Radon measure $\mu = \mu_1-\mu_2$ with a total variation measure $|\mu|=\mu_1+\mu_2$, where $\mu_i,i=1,2$ are non-negative measures. 

Throughout the article, we also use the short notation
$$
\varphi(a) = \E (Y\1_{Y>a}) = \frac{1}{\sqrt{2\pi}}e^{-\frac{a^2}{2}},
$$
where $Y \sim N(0,1)$. 

Our main result is the following.
\begin{theorem}
\label{thm:main}
Let $\Psi$ be a convex function with the left sided derivative $\Psi'$ and let $\mu$ denote the measure associated to the second derivative of $\Psi$ such that $\int_{\mathbb{R}}\varphi(a)\mu(da) < \infty$. Let $X$ be a Gaussian process as above. Then 
\begin{equation}
\label{eq:main-thm}
\begin{split}
&\E\left|\int_0^1 \Psi'(X_s)dX_s - \sum_{k=1}^n \Psi'(X_{t_{k-1}})(X_{t_k}-X_{t_{k-1}})\right| \\
&= \sigma^2\int_{\mathbb{R}}\int_0^1 \frac{1}{\sqrt{V(s)}} \varphi \left(\frac{a}{\sqrt{V(s)}}\right) ds\mu(da) \left(\frac{1}{n}\right)^{2H-1} + \int_{\mathbb{R}}R_n(a)\mu(da),
\end{split}
\end{equation}
where the remainder satisfies
$$
\int_{\mathbb{R}}R_n(a)\mu(da) \leq C\max\{n^{-H},n^{1-2H}\max_{1\leq k\leq n} [g(t_k,t_{k-1})n^{2H}]\}
$$
for some constant $C$ depending solely on the variance function $V(s)$.
\end{theorem}
\begin{remark}
It follows from assumption $\int_{\mathbb{R}}\varphi(a)\mu(da) < \infty$ that the random objects in \eqref{eq:main-thm} are integrable, and hence the bound \eqref{eq:main-thm} makes sense. Indeed, by the proof of Theorem \ref{thm:main} we obtain that the difference of the stochastic integral and its approximation in \eqref{eq:main-thm} is integrable. Moreover, in view of \eqref{eq:ito} and Lemma \ref{lemma:integrability} below, it follows that stochastic integral is integrable. These facts imply that the Riemann approximation in \eqref{eq:main-thm} is integrable as well. 
\end{remark}
For functions of locally bounded variation we obtain immediately the following corollary.
\begin{corollary}
\label{cor:main}
Let $\Psi'$ be of locally bounded variation with $|\mu|$ as its total variation measure. Suppose  $\int_{\mathbb{R}}\varphi(a)|\mu|(da) < \infty$ and let $X$ be a Gaussian process as above. Then 
\begin{equation*}
\begin{split}
&\E\left|\int_0^1 \Psi'(X_s)dX_s - \sum_{k=1}^n \Psi'(X_{t_{k-1}})(X_{t_k}-X_{t_{k-1}})\right| \\
&\leq  \sigma^2\int_{\mathbb{R}}\int_0^1 \frac{1}{\sqrt{V(s)}} \varphi \left(\frac{a}{\sqrt{V(s)}}\right) ds|\mu|(da) \left(\frac{1}{n}\right)^{2H-1} + \int_{\mathbb{R}}R_n(a)|\mu|(da),
\end{split}
\end{equation*}
where the remainder satisfies
$$
\int_{\mathbb{R}}R_n(a)|\mu|(da) \leq C\max\{n^{-H},n^{1-2H}\max_{1\leq k\leq n} [g(t_k,t_{k-1})n^{2H}]\}
$$
for some constant $C$ depending solely on the variance function $V(s)$.
\end{corollary}
Finally, as a by-product of our proof we obtain lower and upper bounds with a weaker condition on the variogram $\vartheta(t,s)$. 
\begin{corollary}
\label{cor:main2}
Let $\Psi$ be a convex function with the left sided derivative $\Psi'$ and let $\mu$ denote the measure associated to the second derivative of $\Psi$ such that $\int_{\mathbb{R}}\varphi(a)\mu(da) < \infty$. Let $X$ be a centered Gaussian process with a non-decreasing variance function $V(s)$ with $V(1)=1$. Suppose further that the variogram satisfies 
$$
\sigma_-^2|t-s|^{2H} \leq \vartheta(t,s)\leq \sigma_+^2|t-s|^{2H}
$$
for some $H\in\left(\frac12,1\right)$. 
Then there exist constants $C_-$ and $C_+$ such that
\begin{equation*}
\begin{split}
&C_- \int_{\mathbb{R}}\int_0^1 \frac{1}{\sqrt{V(s)}} \varphi \left(\frac{a}{\sqrt{V(s)}}\right) ds\mu(da)\left(\frac{1}{n}\right)^{2H-1} \\
&\leq 
\E\left|\int_0^1 \Psi'(X_s)dX_s - \sum_{k=1}^n \Psi'(X_{t_{k-1}})(X_{t_k}-X_{t_{k-1}})\right| \\
&\leq C_+ \int_{\mathbb{R}}\int_0^1 \frac{1}{\sqrt{V(s)}} \varphi \left(\frac{a}{\sqrt{V(s)}}\right) ds\mu(da) \left(\frac{1}{n}\right)^{2H-1}.
\end{split}
\end{equation*}
\end{corollary}
\begin{remark}
Note that here we have incorporated the remainders into the constants $C_-$ and $C_+$. If one considers only the leading order terms (with respect to $n$), then $C_- =\sigma_-^2$ and $C_+ = \sigma_+^2$.
\end{remark}
\section{Examples}\label{sect:exa}

Our results cover many interesting Gaussian processes and functions $\Psi'$. First of all, the assumption $\int_{\mathbb{R}}\varphi(a)|\mu|(da)<\infty$ is not very restrictive, due to the exponential decay of $\varphi(a) = \frac{1}{\sqrt{2\pi}}e^{-\frac{a^2}{2}}$. Our Assumption \eqref{eq:variogram-assumption} on the Gaussian process is not very restrictive either as the following examples show.

\begin{example}
The normalized multi-mixed fractional Brownian motion (see \cite{Almani-Sottinen-2021}) is the process
$$
X_t = \sum_{k=1}^n \sigma_k B^{H_k},
$$
where $\sum_{k=1}^n \sigma_k^2 =1$ and $B^{H_k}$'s are independent fractional Brownian motions with Hurst indices $H_k$.  Let $H_{\min} = \min_{k\le n} H_k$ and let $k_{\min}$ be the index of $H_{\min}$ (here we assume for the sake of simplicity that $k_{\min}$ is unique). Assume that $H_{\min}>\frac12$. We have
$$
\vartheta(t,s) = \sigma_{k_{\min}}|t-s|^{2H_{\min}} + g(t,s),
$$
where
$$
g(t,s) = \sum_{k\ne k_{\min}} \sigma_k^2|t-s|^{2H_k}.
$$
Theorem \ref{thm:main} is applicable with rate $H_{\min}$ and 
$V(s) = \sum_{k=1}^n \sigma_k s^{2H_k}$.
\end{example}	

\begin{example}
Let $X$ be a centered stationary Gaussian process with covariance function $r$ satisfying, for some $H\in\left(\frac12,1\right)$,
$$
r(0) -r(t) = \sigma^2|t|^{2H} + g(t),
$$
where $\frac{g(t)}{|t|^{2H}} \to 0$ as $t\to 0$. Theorem \ref{thm:main} is applicable with rate $H$ and variance function $V(s) = V(0)$. This example covers many interesting stationary Gaussian processes, including fractional Ornstein-Uhlenbeck and related processes (see, e.g. \cite{Istas-Lang,Kaarakka-Salminen-2011}).
\end{example}

\begin{example}
The normalized sub-fractional Brownian $S^H$ motion with index $H\in(0,1)$ (see \cite{Bojdecki-Gorostiza-Talarczyk-2004}) is a centered Gaussian process with covariance
$$
R(t,s) = \sigma^2\left(s^{2H} + t^{2H} - \frac{1}{2}\left((s+t)^{2H} + (s-t)^{2H}\right)\right),
$$	
where $\sigma^2 = 1/(2-2^{2H-1})$ is a normalizing constant. We have
$$
c |t-s|^{2H} \le
\E(S_t^H-S_t^H)^2 \le
C|t-s|^{2H}.
$$ 
Assume that $H>\frac12$. Now Corollary \ref{cor:main2} is applicable with rate $H$ and $V(s) = s^{2H}$.
\end{example}	

\begin{example}
The bifractional Brownian motion (see \cite{Houdre-Villa-2003,Russo-Tudor-2006}) $B^{H,K}$ with indices $H\in(0,1)$ and $K\in(0,1]$ is the centered Gaussian process with covariance
$$
R(t,s) = \frac{1}{2^K}\left(\left(t^{2H}+s^{2H}\right)^K - |t-s|^{2HK}\right).
$$
Similarly to the case of sub-fractional Brownian motion we have
$$
2^{-K} |t-s|^{2HK} \le
\E(B_t^{H,K}-B_t^{H,K})^2 \le
2^{1-K}|t-s|^{2HK}.
$$ 
Assume $HK>\frac12$. Now Corollary \ref{cor:main2} is applicable with rate $HK$ and $V(s) = s^{2HK}$.
\end{example}

\begin{example}
The tempered fractional Brownian motion (see \cite{Azmoodeh-Mishura-Sabzikar}) $X^{H}$ with index $H\in(0,1)$ is the centered Gaussian process with covariance
$$
R(t,s) = \frac{1}{2}\left(C_t^2t^{2H}+C_s^2s^{2H}-C_{t-s}^2|t-s|^{2H}\right)
$$
with a certain function $C_t$ (see \cite[Lemma 2.3]{Azmoodeh-Mishura-Sabzikar}). 
Similarly to the case of sub-fractional and bifractional Brownian motion we have (see \cite[Theorem 2.7]{Azmoodeh-Mishura-Sabzikar})
$$
\sigma^2_-|t-s|^{2H} \le
\E(X_t^{H}-X_t^{H})^2 \le
\sigma^2_+|t-s|^{2H}.
$$ 
Assume $H>\frac12$. Now Corollary \ref{cor:main2} is applicable with rate $H$ and $V(s) = C_s^2s^{2H}$.
\end{example}

\section{Proofs}\label{sect:proofs}
In what follows, $C$ denotes a generic constant that depends only on the variance function $V(s)$, but may vary from line to line.
\subsection{Auxiliary lemmas on Gaussian process $X$ and convex function $\Psi$}
The following is one of our key lemmas and allows to reduce our analysis to the simple case $\Psi(x) =(x-a)^+$. 
\begin{lemma}
\label{lem:convex-dif}
Let $\Psi$ be convex and $\psi = \Psi'_-$ be its left-sided derivative. Then for any $x,y\in \mathbb{R}$ we have
\begin{equation*}
\begin{split}
\Psi(x)- \Psi(y) - \psi(y)(x-y) &= \int_{\mathbb{R}}|x-a|-|y-a|-sgn(y-a)(x-y)\mu(da) \\
&= 2\int_{\mathbb{R}}(x-a)^+-(y-a)^+-\1_{y>a}(x-y)\mu(da)\\
& \geq 0.
\end{split}
\end{equation*}
\end{lemma}
\begin{proof}
Let $I$ be an interval such that $x,y \in I$. Then it is well-known that we have representations \cite{Revuz-Yor}
$$
\Psi(x) = \alpha_I + \beta_Ix + \int_{I} |x-a|\mu(da)
$$
and
$$
\Psi'(x) = \beta_I + \int_I sgn(x-a)\mu(da).
$$
Using these, $|x-a|= 2(x-a)^+ -(x-a)$, and $sgn(y-a) = 2\1_{y>a} - 1$, we obtain that linear terms vanish and we get
\begin{equation*}
\begin{split}
\Psi(x)- \Psi(y) - \psi(y)(x-y) &= \int_I |x-a|-|y-a|-sgn(y-a)(x-y)\mu(da)\\
&= 2\int_{I}(x-a)^+-(y-a)^+-\1_{y>a}(x-y)\mu(da).
\end{split}
\end{equation*}
It is an easy exercise to check that $(x-a)^+-(y-a)^+-\1_{y>a}(x-y)\geq 0$ from which it follows that $\Psi(x)- \Psi(y) - \psi(y)(x-y) \geq 0$ for any convex function $\Psi$. It remains to note that 
$$
\int_{I}(x-a)^+-(y-a)^+-\1_{y>a}(x-y)\mu(da) = \int_{\mathbb{R}}(x-a)^+-(y-a)^+-\1_{y>a}(x-y)\mu(da),
$$
where the latter integral is well-defined since $(x-a)^+-(y-a)^+-\1_{y>a}(x-y) = 0$ whenever $a\notin I$. 
\end{proof}
As a consequence we obtain the following lemma providing us integrability. 
\begin{lemma}
\label{lemma:integrability}
Let $\Psi$ be a convex function with the associated measure $\Psi'' = \mu$ and let $Y \sim N(0,1)$. If $\int_{\mathbb{R}}\varphi(a)\mu(da) < \infty$, then $\E|\Psi(Y)| < \infty$. 
\end{lemma}
\begin{proof}
By adding a linear function if necessary, we may assume without loss of generality that $\Psi\geq 0$. Now 
from Lemma \ref{lem:convex-dif} we deduce that, for any deterministic $z$,
$$
\Psi(Y) - \Psi(z) - \Psi'_-(z)(Y-z) = 2\int_{\mathbb{R}} \left[(Y-a)^+ - (z-a)^+ - \1_{z>a}(Y-z)\right]\mu(da).
$$
Taking expectation and using Tonelli's theorem we get 
$$
\E \Psi(Y) - \Psi(z) + \Psi'_-(z)z = 2\int_{\mathbb{R}}\left[\E (Y-a)^+ - (z-a)^+ + \1_{z>a}z\right]\mu(da).
$$
In particular, for $z=0$ we get 
$$
\E \Psi(Y) - \Psi(0) = 2\int_{\mathbb{R}}\left[\E(Y-a)^+ -(-a)^+\right]\mu(da).
$$
Hence it suffices to prove 
$$
\E(Y-a)^+ -(-a)^+ \leq C\varphi(a).
$$
However, this now follows by observing that 
$$
\E (Y-a)^+ -(-a)^+=  \varphi(a) - a\textbf{P}(Y>a)-(-a)^+ = \varphi(a)-|a|P(Y>|a|)
$$
and the well-known asymptotical relation $a\textbf{P}(Y>a) \sim \varphi(a)$. 
\end{proof}
Next we establish several lemmas related to the Gaussian process $X$.
\begin{lemma}
\label{lem:variance-bounds}
We always have
$$
\sqrt{V(t_k)} - \sqrt{V(t_{k-1})} \leq \sqrt{\vartheta(t_k,t_{k-1})} \leq Cn^{-H}
$$
and 
\begin{equation}
\label{eq:variance-quotient-bounded}
\sup_{n\geq 1}\sup_{2\leq k\leq n} \frac{\sqrt{V(t_k)}}{\sqrt{V(t_{k-1})}} < \infty.
\end{equation}
\end{lemma}
\begin{proof}
By Gaussianity we have $\E|X_t| = C\sqrt{V(t)}$ from which reverse triangle inequality gives
$$
\sqrt{V(t_k)} - \sqrt{V(t_{k-1})} = C\E|X_{t_k}| - C\E|X_{t_{k-1}}| \leq C\E|X_t-X_s|
$$
leading to the first claim. The second claim now follows from
$$
\frac{\sqrt{V(t_k)}}{\sqrt{V(t_{k-1})}} = 1 + \frac{\sqrt{\vartheta(t_k,t_{k-1})}}{\sqrt{V(t_{k-1})}}
$$
and the fact that $V(t_{k-1}) \geq V(t_1) \geq cn^{-H}$. 
\end{proof}
Throughout, we use the following short notation 
$$
\gamma_k = \frac{R(t_k,t_{k-1})}{V(t_{k-1})},
$$
where $R(t,s)$ is the covariance function of $X$, and we use the convention $\gamma_k=0$ whenever $V(t_{k-1})=0$. The following gives us a useful relation.
\begin{lemma}
\label{lem:Gaussian-function-relations}
Let $V(t_{k-1})>0$. Then 
$$\sqrt{V(t_k)} - \gamma_k\sqrt{V(t_{k-1})} = -\frac{\left(\sqrt{V(t_k)} - \sqrt{V(t_{k-1})}\right)^2}{2\sqrt{V(t_{k-1})}} + \frac{\vartheta(t_k,t_{k-1})}{2V(t_{k-1})}.
$$
\end{lemma}
\begin{proof}
We use 
$$
\sqrt{V(t_k)} - \gamma_k\sqrt{V(t_{k-1})} = \sqrt{V(t_k)} - \sqrt{V(t_{k-1})} + \left[1-\gamma_k\right]\sqrt{V(t_{k-1})}
$$
and 
$$
\gamma_k - 1 = \frac{V(t_k)-V(t_{k-1}) - \vartheta(t_k,t_{k-1})}{2V(t_{k-1})}. 
$$
Using also 
\begin{equation*}
\begin{split}
V(t_k)-V(t_{k-1}) &= \left(\sqrt{V(t_k)} - \sqrt{V(t_{k-1})}\right)\left(\sqrt{V(t_k)} + \sqrt{V(t_{k-1})}\right) \\
&= \left(\sqrt{V(t_k)} - \sqrt{V(t_{k-1})}\right)^2 + 2\left(\sqrt{V(t_k)} - \sqrt{V(t_{k-1})}\right)\sqrt{V(t_{k-1})}
\end{split}
\end{equation*}
leads to
\begin{equation*}
\begin{split}
\left[1-\gamma_k\right]\sqrt{V(t_{k-1})} &= -\frac{\left(\sqrt{V(t_k)} - \sqrt{V(t_{k-1})}\right)^2}{2\sqrt{V(t_{k-1})}} - \left(\sqrt{V(t_k)} - \sqrt{V(t_{k-1})}\right) \\
&+ \frac{\vartheta(t_k,t_{k-1})}{2V(t_{k-1})}.
\end{split}
\end{equation*}
Consequently, we have
$$
\sqrt{V(t_k)} - \gamma_k\sqrt{V(t_{k-1})} = -\frac{\left(\sqrt{V(t_k)} - \sqrt{V(t_{k-1})}\right)^2}{2\sqrt{V(t_{k-1})}} + \frac{\vartheta(t_k,t_{k-1})}{2V(t_{k-1})},
$$
completing the proof.
\end{proof}

\subsection{Approximation estimates}
We begin with the following elementary lemma on the approximation of Riemann-Stieltjes integrals. For the reader's convenience, we present the proof.
\begin{lemma}
\label{lemma:RS-integral-bound}
Let $f$ be a differentiable function on $[0,1]$ and let $g$ be non-decreasing on $[0,1]$. Then
\begin{equation*}
\begin{split}
&\left|\int_0^1 f(V(s))dg(s)-\sum_{k=1}^n f(V(t_{k-1}))(g(t_k)-g(t_{k-1}))\right| \\
&\leq \max_{1\leq k\leq n}(g(t_k)-g(t_{k-1}))\int_0^1 |f'(s)|ds.
\end{split}
\end{equation*}
\end{lemma}
\begin{proof}
Without loss of generality, we can assume $\int_0^1 |f'(s)|ds < \infty$ since otherwise there is nothing to prove. From this it follows that $f$ is of bounded variation, since for a differentiable function we have 
$$
TV(f) = \int_0^1 |f'(s)|ds,
$$
where $TV$ stands for total variation.
Since $V$ is continuous and non-decreasing, this further implies that $f(V(\cdot))$ is continuous and of bounded variation as well, with 
$$
TV(f(V)) \leq \int_0^1 |f'(s)|ds.
$$
Indeed, this follows from the fact that 
\begin{equation*}
\begin{split}
TV(f(V)) &=\sup_{\{s_1,s_2,\ldots,s_n\}} \sum_{k=1}^n |f(V(s_k))-f(V(s_{k-1}))| \\
&\leq \sup_{\{x_1,x_2,\ldots,x_n\}}\sum_{k=1}^n |f(x_k)-f(x_{k-1})| = TV(f).
\end{split}
\end{equation*} 
Thus the Riemann-Stieltjes integral $\int_0^1 f(V(s))dg(s)$ exists, as $f(V(s))$ is continuous and $g(s)$ is non-decreasing, and hence of bounded variation. Let us now prove the claimed upper bounds. We have
\begin{equation*}
\begin{split}
&|\int_0^1 f(V(s))dg(s)-\sum_{k=1}^n f(V(t_{k-1}))(g(t_k)-g(t_{k-1}))| \\
&\leq \sum_{k=1}^n \int_{t_{k-1}}^{t_k} |f(V(s))-f(V(t_{k-1}))|dg(s) \\
&\leq \sum_{k=1}^n |f(V(s_k^*))-f(V(t_{k-1}))|[g(t_k)-g(t_{k-1})] \\
&\leq \max_{1\leq k \leq n} [g(t_k)-g(t_{k-1})] \sup_{\{s_1,s_2,\ldots,s_n\}} \sum_{k=1}^n |f(V(s_k))-f(V(s_{k-1}))|\\
&\leq \max_{1\leq k\leq n} [g(t_k)-g(t_{k-1})] \int_0^1 |f'(s)|ds
\end{split}
\end{equation*}
proving the claimed upper bound. This completes the proof.
\end{proof}
We apply the result for function $f(x) = \frac{1}{\sqrt{x}}e^{-\frac{a^2}{2x}}$. The following lemma evaluates the integral for this function in terms of the level $a$ when the level $a$ is large enough.
\begin{lemma}
\label{lem:total-var-a-big}
Let $|a|>1$. Then for $f(x) = \frac{1}{\sqrt{x}}e^{-\frac{a^2}{2x}}$ we have
$$
\int_0^1 |f'(s)|ds \leq C\varphi(a).
$$
\end{lemma}
\begin{proof}
By straightforward computations we get 
$$
f'(x) = \frac{1}{2}e^{-\frac{a^2}{2x}}x^{-\frac52}(a^2-x)
$$
from which we get 
$$
|f'(x)| = \frac{1}{2}e^{-\frac{a^2}{2x}}x^{-\frac52}(a^2-x)
$$ 
as $x\in [0,1]$ and $|a|>1$. Now 
\begin{equation*}
\begin{split}
\int_0^1 |f'(s)|ds &\leq \int_0^1 \frac{1}{2}e^{-\frac{a^2}{2s}}s^{-\frac52}a^2ds \\
&= \frac{a^2}{2}\int_{\frac{a^2}{2}}^\infty e^{-z}\left(\frac{a^2}{2z}\right)^{-\frac52}\frac{a^2}{2z^2}dz \\
&= \frac{\sqrt{2}}{a} \int_{\frac{a^2}{2}}^\infty e^{-z}\sqrt{z}dz.
\end{split}
\end{equation*}
By L'Hopital's rule, we obtain that
$$
\lim_{a\to \infty}\frac{\int_{\frac{a^2}{2}}^\infty e^{-z}\sqrt{z}dz}{ae^{-\frac{a^2}{2}}} = \lim_{a\to \infty} \frac{e^{-\frac{a^2}{2}}\frac{a}{\sqrt{2}}\cdot a}{a^2e^{-\frac{a^2}{2}}-e^{-\frac{a^2}{2}}} = \frac{1}{\sqrt{2}}.
$$
It follows that
$$
\int_0^1 |f'(s)|ds \leq \frac{C}{a} \cdot a e^{-\frac{a^2}{2}} = C\varphi(a).
$$
This completes the proof.
\end{proof}
The following lemma is to obtain boundedness in the region $|a|\leq 1$.
\begin{lemma}
\label{lemma:boundedness-a-small}
Set $f_a(x) = \frac{a^4}{x^2}e^{-\frac{a^2}{2x}}$. Then
$$
\sup_{|a|\leq 1}\sup_{0\leq x\leq 1} f_a(x) < \infty.
$$
\end{lemma}
\begin{proof}
The claim follows directly by noting that $f_a(x) = h\left(\frac{a^2}{x}\right)$, where 
$$
h(z) = z^2e^{-\frac{z}{2}}
$$
is bounded for $z\geq 0$. 
\end{proof}
\begin{lemma}
\label{lemma:double-difference}
We have, for $|a|\leq 1$,
$$
\sum_{k=2}^n \frac{1}{\sqrt{V(t_{k-1})}}\left[\varphi\left(\frac{a^2}{\sqrt{V(t_{k-1})}}\right)-\varphi\left(\frac{a^2}{\sqrt{V(t_{k})}}\right)\right](t_k-t_{k-1}) \leq C\varphi(a)n^{-H}.
$$
\end{lemma}
\begin{proof}
By mean value theorem and the fact $\varphi'(x) = -x\varphi(x)$ we have 
\begin{equation*}
\begin{split}
&\frac{1}{\sqrt{V(t_{k-1})}}\left[\varphi\left(\frac{a^2}{\sqrt{V(t_{k-1})}}\right)-\varphi\left(\frac{a^2}{\sqrt{V(t_{k})}}\right)\right] \\
&\leq \frac{1}{\sqrt{V(t_{k-1})}}\left(\frac{a^2}{\sqrt{V(t_k)}}-\frac{a^2}{\sqrt{V(t_{k-1})}}\right)\frac{a^2}{\sqrt{\xi_k}}\varphi\left(\frac{a^2}{\sqrt{\xi_k}}\right)\\
&\leq \frac{\xi_k^{\frac32}}{\sqrt{V(t_{k})}V(t_{k-1})}\Delta_k \sqrt{V(\cdot)}\frac{a^4}{\xi_k^2}\varphi\left(\frac{a^2}{\sqrt{\xi_k}}\right).
\end{split}
\end{equation*}
Here 
$$
\sup_k \frac{\xi_k^{\frac32}}{\sqrt{V(t_{k})}V(t_{k-1})} < \infty
$$
by Lemma \ref{lemma:boundedness-a-small}, while 
$$
\sup_{k}\sup_{|a|\leq 1}\frac{a^4}{\xi_k^2}\varphi\left(\frac{a^2}{\sqrt{\xi_k}}\right) < \infty
$$
by Lemma \ref{lem:variance-bounds}. The claim follows from $\Delta_k\sqrt{V(\cdot)} \leq Cn^{-H}$.
\end{proof}
\begin{lemma}
\label{lemma:interpolate-integral}
We have
\begin{equation*}
\begin{split}
&\left|\sum_{k=2}^{n-1}\left[V(t_{k-1})\right]^{-\frac12} \varphi\left(\frac{a^2}{\sqrt{V(t_k)}}\right)[t_k-t_{k-1}]-\int_0^1\left[V(s)\right]^{-\frac12} \varphi\left(\frac{a^2}{\sqrt{V(s)}}\right)ds\right| \\
&\leq C\varphi(a)n^{H-1}.
\end{split}
\end{equation*}
\end{lemma}
\begin{proof}
From monotonicity we get
\begin{equation*}
\begin{split}
\int_{t_{k-1}}^{t_{k}} \left[V(s)\right]^{-\frac12} \varphi\left(\frac{a^2}{\sqrt{V(s)}}\right)ds
& \leq \left[V(t_{k-1})\right]^{-\frac12} \varphi\left(\frac{a^2}{\sqrt{V(t_k)}}\right)[t_k-t_{k-1}]\\
&\leq \int_{t_{k}}^{t_{k+1}} \left[V(s)\right]^{-\frac12} \varphi\left(\frac{a^2}{\sqrt{V(s)}}\right)ds.
\end{split}
\end{equation*}
Summing over $k=2,\ldots,n-1$ yields
\begin{equation*}
\begin{split}
\int_{t_{1}}^{t_{n-1}} \left[V(s)\right]^{-\frac12} \varphi\left(\frac{a^2}{\sqrt{V(s)}}\right)ds&\leq \sum_{k=2}^{n-1}\left[V(t_{k-1})\right]^{-\frac12} \varphi\left(\frac{a^2}{\sqrt{V(t_k)}}\right)[t_k-t_{k-1}]\\
&\leq \int_{t_{2}}^{1} \left[V(s)\right]^{-\frac12} \varphi\left(\frac{a^2}{\sqrt{V(s)}}\right)ds
\end{split}
\end{equation*}
from which we get 
\begin{equation*}
\begin{split}
 & \left|\sum_{k=2}^{n-1}\left[V(t_{k-1})\right]^{-\frac12} \varphi\left(\frac{a^2}{\sqrt{V(t_k)}}\right)[t_k-t_{k-1}]-\int_0^1\left[V(s)\right]^{-\frac12} \varphi\left(\frac{a^2}{\sqrt{V(s)}}\right)ds\right| \\
&\leq \int_0^{t_{1}} \left[V(s)\right]^{-\frac12} \varphi\left(\frac{a^2}{\sqrt{V(s)}}\right)ds+ \int_{t_{n-1}}^1\left[V(s)\right]^{-\frac12} \varphi\left(\frac{a^2}{\sqrt{V(s)}}\right)ds. 
\end{split}
\end{equation*}
Here
\begin{equation*}
\begin{split}
&\int_0^{t_{1}} \left[V(s)\right]^{-\frac12} \varphi\left(\frac{a^2}{\sqrt{V(s)}}\right)ds + \int_{t_{n-1}}^1\left[V(s)\right]^{-\frac12} \varphi\left(\frac{a^2}{\sqrt{V(s)}}\right)ds\\
&\leq \varphi(a)\int_0^{t_{1}} \left[V(s)\right]^{-\frac12}ds + \varphi(a)\int_{t_{n-1}}^1 \left[V(s)\right]^{-\frac12}ds\\
&\leq Cn^{H-1}
\end{split}
\end{equation*}
by the fact $V(s) \geq cs^{2H}$.
\end{proof}
\begin{lemma}
\label{lemma:leading-term-bound}
We have
\begin{equation*}
\begin{split}
&\left|\sum_{k=2}^n \frac{1}{2\sqrt{V(t_{k-1})}} \varphi\left(\frac{a}{\sqrt{V(t_{k-1})}}\right)(t_k-t_{k-1}) - \int_0^1 \frac{1}{2\sqrt{V(s)}} \varphi\left(\frac{a}{\sqrt{V(s)}}\right)ds\right| \\
&\leq C\varphi(a)n^{H-1}.
\end{split}
\end{equation*}
\end{lemma}
\begin{proof}
We separate the cases $|a|>1$ and $|a|\leq 1$. Let first $|a|>1$. Noting that then, by using the convention $\frac{1}{x}\varphi\left(\frac{a}{x}\right)=0$ for $x=0$, we have 
\begin{equation*}
\begin{split}
&\sum_{k=2}^n \frac{1}{2\sqrt{V(t_{k-1})}} \varphi\left(\frac{a}{\sqrt{V(t_{k-1})}}\right)(t_k-t_{k-1})\\
& = \sum_{k=1}^n \frac{1}{2\sqrt{V(t_{k-1})}} \varphi\left(\frac{a}{\sqrt{V(t_{k-1})}}\right)(t_k-t_{k-1}).
\end{split}
\end{equation*}
Now Lemma \ref{lemma:RS-integral-bound} and Lemma \ref{lem:total-var-a-big} apply, and we get, with $f(x) = \frac{1}{\sqrt{x}}e^{-\frac{a^2}{2x}}$, that
\begin{equation*}
\begin{split}
&\left|\sum_{k=2}^n \frac{1}{2\sqrt{V(t_{k-1})}} \varphi\left(\frac{a}{\sqrt{V(t_{k-1})}}\right)(t_k-t_{k-1}) - \int_0^1 \frac{1}{2\sqrt{V(s)}} \varphi\left(\frac{a}{\sqrt{V(s)}}\right)ds\right| \\
&\leq \frac{\int_0^1\left|f'(s)\right|ds}{n} \leq \frac{C\varphi(a)}{n} \leq C\varphi(a)n^{H-1}.
\end{split}
\end{equation*}
This proves the claim when $|a|>1$. For $|a|\leq 1$, we write
\begin{equation*}
\begin{split}
&\sum_{k=2}^n \frac{1}{2\sqrt{V(t_{k-1})}} \varphi\left(\frac{a}{\sqrt{V(t_{k-1})}}\right)(t_k-t_{k-1}) - \int_0^1 \frac{1}{2\sqrt{V(s)}} \varphi\left(\frac{a}{\sqrt{V(s)}}\right)ds \\
&= \sum_{k=2}^n \frac{1}{2\sqrt{V(t_{k-1})}} \varphi\left(\frac{a}{\sqrt{V(t_{k})}}\right)(t_k-t_{k-1})-\int_0^1 \frac{1}{2\sqrt{V(s)}} \varphi\left(\frac{a}{\sqrt{V(s)}}\right)ds\\
&+\sum_{k=2}^n \frac{1}{2\sqrt{V(t_{k-1})}} \left[\varphi\left(\frac{a}{\sqrt{V(t_{k-1})}}\right)-\varphi\left(\frac{a}{\sqrt{V(t_{k})}}\right)\right](t_k-t_{k-1}).
\end{split}
\end{equation*}
The second term can be bounded by Lemma \ref{lemma:double-difference} and we have 
\begin{equation*}
\begin{split}
&\sum_{k=2}^n \frac{1}{2\sqrt{V(t_{k-1})}} \left[\varphi\left(\frac{a}{\sqrt{V(t_{k-1})}}\right)-\varphi\left(\frac{a}{\sqrt{V(t_{k})}}\right)\right](t_k-t_{k-1}) \leq Cn^{-H}\\
& \leq C\varphi(a)n^{H-1}
\end{split}
\end{equation*}
since for $|a|\leq 1$ we have $\varphi(a)>\epsilon$.
For the first term, we have by Lemma \ref{lemma:interpolate-integral} that
\begin{equation*}
\begin{split}
&\left|\sum_{k=2}^{n-1} \frac{1}{2\sqrt{V(t_{k-1})}} \varphi\left(\frac{a}{\sqrt{V(t_{k})}}\right)(t_k-t_{k-1})-\int_0^1 \frac{1}{2\sqrt{V(s)}} \varphi\left(\frac{a}{\sqrt{V(s)}}\right)ds\right|\\
&\leq C\varphi(a)n^{H-1}
\end{split}
\end{equation*}
yielding
\begin{equation*}
\begin{split}
&\left|\sum_{k=2}^{n} \frac{1}{2\sqrt{V(t_{k-1})}} \varphi\left(\frac{a}{\sqrt{V(t_{k})}}\right)(t_k-t_{k-1})-\int_0^1 \frac{1}{2\sqrt{V(s)}} \varphi\left(\frac{a}{\sqrt{V(s)}}\right)ds\right| \\
&\leq C\varphi(a)n^{H-1} + \frac{1}{2\sqrt{V(t_{n-1})}} \varphi\left(\frac{a}{\sqrt{V(1)}}\right)n^{-1}\\
&\leq C\varphi(a)n^{H-1}.
\end{split}
\end{equation*}
This proves the case $|a|\leq 1$ and completes the whole proof.
\end{proof}
\subsection{Proof of Theorem \ref{thm:main} and Corollary \ref{cor:main}}
We begin by considering a simple case $f(x) = (x-a)^+$. 
\begin{proposition}
\label{prop:main}
Let $a\in \mathbb{R}$ be fixed. Then
\begin{equation*}
\begin{split}
&\E \left|\int_0^1 I_{X_s>a}dX_s - \sum_{k=1}^n \1_{X_{t_{k-1}}>a}(X_{t_k}-X_{t_{k-1}})\right|\\
& = \frac12 \int_0^1 \frac{1}{\sqrt{V(s)}} \varphi \left(\frac{a}{\sqrt{V(s)}}\right) ds \left(\frac{1}{n}\right)^{2H-1} + R_n(a),
\end{split}
\end{equation*}
where the remainder satisfies 
$$
R_n(a) \leq C\varphi(a)\max\{n^{-H},n^{1-2H}\max_{1\leq k\leq n} [g(t_k,t_{k-1})n^{2H}]\}.
$$
\end{proposition}
\begin{proof}
By \eqref{eq:ito} we have
$$
\int_0^1 I_{X_s>a}dX_s = (X_1-a)^+-(X_0-a)^+.
$$
 Writing 
$$
(X_1-a)^+-(X_0-a)^+ = \sum_{k=1}^n \left[(X_{t_k}-a)^+ - (X_{t_{k-1}}-a)^+\right],
$$
we get
\begin{equation*}
\begin{split}
&(X_1-a)^+-(X_0-a)^+ - \sum_{k=1}^n \1_{X_{t_{k-1}}>a}(X_{t_k}-X_{t_{k-1}}) \\
&= \sum_{k=1}^n \left[(X_{t_k}-a)^+ - (X_{t_{k-1}}-a)^+ - \1_{X_{t_{k-1}}>a}(X_{t_k}-X_{t_{k-1}})\right]\\
&\geq 0,
\end{split}
\end{equation*}
where the last inequality follows from Lemma \ref{lem:convex-dif}. From $
(x-a)^+ = x\1_{x>a} - a \1_{x>a}
$ we obtain for one interval increment
\begin{equation*}
\begin{split}
&(X_{t_k}-a)^+ - (X_{t_{k-1}}-a)^+ - \1_{X_{t_{k-1}}>a}(X_{t_k}-X_{t_{k-1}})\\
&= X_{t_k}\1_{X_{t_k}>a} - X_{t_k}\1_{X_{t_{k-1}}>a} -a \1_{X_{t_k} > a} + a\1_{X_{t_{k-1}}>a}.
\end{split}
\end{equation*}
If $V(t_{k-1})>0$, using representation 
$$
X_{t_k} = \frac{R(t_k,t_{k-1})}{V(t_{k-1})}X_{t_{k-1}} + bY,
$$
where $Y\sim N(0,1)$ is independent of $X_{t_{k-1}}$ and $b$ is such that $\E X_{t_k}^2 = V(t_k)$,
we get
$$
\E \left(X_{t_k} \1_{X_{t_{k-1}}>a}\right) = \frac{R(t_k,t_{k-1})}{V(t_{k-1})}\E \left(X_{t_{k-1}}\1_{X_{t_{k-1}}>a}\right) = \gamma_k \sqrt{V(t_{k-1})} \varphi\left(\frac{a}{\sqrt{V(t_{k-1})}}\right).
$$
After rearranging the terms this leads to 
\begin{equation*}
\begin{split}
&\E \left[(X_{t_k}-a)^+ - (X_{t_{k-1}}-a)^+ - \1_{X_{t_{k-1}}>a}(X_{t_k}-X_{t_{k-1}})\right] \\
&= \sqrt{V(t_k)}\varphi\left(\frac{a}{\sqrt{V(t_k)}}\right) - \gamma_k\sqrt{V(t_{k-1})}\varphi\left(\frac{a}{\sqrt{V(t_{k-1})}}\right)\\
& + a \pr\left(Y>\frac{a}{\sqrt{V(t_{k-1})}}\right) - a \pr\left(Y>\frac{a}{\sqrt{V(t_{k})}}\right) \\
&= \left[\sqrt{V(t_k)} - \gamma_k\sqrt{V(t_{k-1})}\right] \varphi\left(\frac{a}{\sqrt{V(t_{k-1})}}\right) \\
&+ \sqrt{V(t_{k})}\left[\varphi\left(\frac{a}{\sqrt{V(t_k)}}\right)- \varphi\left(\frac{a}{\sqrt{V(t_{k-1})}}\right)\right]\\
& + a \pr\left(Y>\frac{a}{\sqrt{V(t_{k-1})}}\right) - a \pr\left(Y>\frac{a}{\sqrt{V(t_{k})}}\right).
\end{split}
\end{equation*}
Note also that this remains valid in the case when $V(t_{k-1})=0$, provided we use the convention $\pr(Y>\infty) = 0$, $\pr(Y>-\infty)=1$, $\varphi(\pm \infty)=0$, and 
$$
\gamma_k \sqrt{V(t_{k-1})} \varphi\left(\frac{a}{\sqrt{V(t_{k-1})}}\right) = 0.
$$
We have obtained 
$$
\E \left|(X_1-a)^+-(X_0-a)^+ - \sum_{k=1}^n \1_{X_{t_{k-1}}>a}(X_{t_k}-X_{t_{k-1}})\right| =  I_{0,n} +I_{1,n} + I_{2,n}+I_{3,n}, 
$$
where 
$$
I_{0,n} = \left[\sqrt{V(t_1)} - \gamma_1\sqrt{V(0)}\right] \varphi\left(\frac{a}{\sqrt{V(0)}}\right),
$$
$$
I_{1,n} = \sum_{k=2}^n\left[\sqrt{V(t_k)} - \gamma_k\sqrt{V(t_{k-1})}\right] \varphi\left(\frac{a}{\sqrt{V(t_{k-1})}}\right),
$$
$$
I_{2,n} = \sum_{k=1}^n\sqrt{V(t_{k})}\left[\varphi\left(\frac{a}{\sqrt{V(t_k)}}\right)- \varphi\left(\frac{a}{\sqrt{V(t_{k-1})}}\right)\right],
$$
and 
$$
I_{3,n} = \sum_{k=1}^n\left[a \pr\left(Y>\frac{a}{\sqrt{V(t_{k-1})}}\right) - a \pr\left(Y>\frac{a}{\sqrt{V(t_{k})}}\right)\right].
$$
For $I_{0,n}$ we have 
$$
|I_{0,n}| \leq \varphi(a)\left|\sqrt{V(t_1)} - \gamma_1\sqrt{V(0)}\right|.
$$
Here $\gamma_1 = 0$ if $V(0)=0$ leading to $|I_{0,n}| \leq \varphi(a) n^{-H}$, while for $V(0)>0$ we can use Lemma \ref{lem:Gaussian-function-relations} and Lemma \ref{lem:variance-bounds} to obtain
\begin{equation*}
\begin{split}
&\left|\sqrt{V(t_1)} - \gamma_1\sqrt{V(0)}\right| \\
&\leq \frac{\left(\sqrt{V(t_1)} - \sqrt{V(0)}\right)^2}{2\sqrt{V(0)}} + \frac{\vartheta(t_1,t_0)}{2V(0)} \\
&\leq Cn^{-2H}
\end{split}
\end{equation*}
leading to $|I_{0,n}|\leq C\varphi(a)n^{-H}$ as well. Consider next the terms $I_{2,n}$ and $I_{3,n}$.
Trivially 
$$
I_{3,n} = a \pr\left(Y>\frac{a}{\sqrt{V(0)}}\right) - a \pr\left(Y>\frac{a}{\sqrt{V(1)}}\right),
$$
while for $I_{2,n}$ we get by Lemma \ref{lemma:RS-integral-bound} for each subinterval that 
$$
I_{2,n} = \int_0^1 \sqrt{V(s)}d\varphi\left(\frac{a}{\sqrt{V(s)}}\right) + R_{2,n},
$$
where the remainder satisfies, since $\varphi\left(\frac{a}{\sqrt{V(s)}}\right)$ is increasing in $s$, 
$$
R_{2,n} \leq \max_{1\leq k\leq n} \Delta_k \sqrt{V(\cdot)}\varphi(a) \leq C\varphi(a) n^{-H}.
$$
Note that here, by using the fact that $\varphi'(x)=-x\varphi(x)$ and $\varphi(x)$ is the density of the normal distribution, 
\begin{equation*}
\begin{split}
&\int_0^1 \sqrt{V(s)}d\varphi\left(\frac{a}{\sqrt{V(s)}}\right) \\
&=-\int_0^1 \sqrt{V(s)}\varphi\left(\frac{a}{\sqrt{V(s)}}\right)\left(\frac{a}{\sqrt{V(s)}}\right)a d\left[\left(V(s)\right)^{-\frac12}\right]\\
&=-a^2\int_0^1 \varphi\left(\frac{a}{\sqrt{V(s)}}\right) d\left[\left(V(s)\right)^{-\frac12}\right]\\
&=a^2 \int_{\frac{1}{\sqrt{V(1)}}}^{\frac{1}{\sqrt{V(0)}}} \varphi(az)dz \\
&=a \int_{\frac{a}{\sqrt{V(1)}}}^{\frac{a}{\sqrt{V(0)}}} \varphi(v)dv\\
&=-a \pr\left(Y>\frac{a}{\sqrt{V(0)}}\right) + a \pr\left(Y>\frac{a}{\sqrt{V(1)}}\right).
\end{split}
\end{equation*}
Consequently, we have
\begin{equation*}
\begin{split}
I_{2,n}+I_{3,n} &= R_{2,n} \leq C\varphi(a)n^{-H}.
\end{split}
\end{equation*}
It remains to bound the term $I_{1,n}$. Using Lemma \ref{lem:Gaussian-function-relations} allows us to write $I_{1,n} = I_{1,A,n} + I_{1,B,n}$, where
$$
I_{1,A,n} = -\sum_{k=2}^n \frac{\left(\sqrt{V(t_k)} - \sqrt{V(t_{k-1})}\right)^2}{2\sqrt{V(t_{k-1})}} \varphi\left(\frac{a}{\sqrt{V(t_{k-1})}}\right)
$$
and
$$
I_{1,B,n} = \sum_{k=2}^n \frac{\vartheta(t_k,t_{k-1})}{2\sqrt{V(t_{k-1})}} \varphi\left(\frac{a}{\sqrt{V(t_{k-1})}}\right).
$$
For $I_{1,A,n}$ we estimate
\begin{equation*}
\begin{split}
|I_{1,A,n}| &\leq \varphi(a)\max_{1\leq k\leq n}\Delta_k \sqrt{V(\cdot)}\sum_{k=2}^n \frac{\sqrt{V(t_k)} - \sqrt{V(t_{k-1})}}{2\sqrt{V(t_{k-1})}}\\
& \leq C\varphi(a) n^{-H} \sum_{k=2}^n \frac{\sqrt{V(t_k)}}{\sqrt{V(t_{k-1})}}\frac{1}{\sqrt{V(t_{k})}}\left(\sqrt{V(t_k)} - \sqrt{V(t_{k-1})}\right)\\
&\leq C\varphi(a)n^{-H}\sum_{k=1}^n \frac{1}{\sqrt{V(t_{k})}}\left(\sqrt{V(t_k)} - \sqrt{V(t_{k-1})}\right)\\
& = C\varphi(a)n^{-H} \sum_{k=1}^n \int_{t_{k-1}}^{t_k}\frac{1}{\sqrt{V(t_{k})}} d\sqrt{V(s)}\\
&\leq C\varphi(a)n^{-H} \int_0^1 \frac{1}{\sqrt{V(s)}}d\sqrt{V(s)} \\
&= C\varphi(a)n^{-H} \int_0^1 dV(s)\\
&=C\varphi(a)n^{-H}.
\end{split}
\end{equation*}
Here we have used the facts that $\frac{1}{\sqrt{V(t_{k})}} \leq \frac{1}{\sqrt{V(s)}}$ for $t_{k-1}\leq s\leq t_k$ as $V$ is non-decreasing, and that $d\sqrt{V(s)} = \frac{1}{2\sqrt{V(s)}}dV(s)$ giving us
$$
\int_0^1 \frac{1}{\sqrt{V(s)}}d\sqrt{V(s)}  = \int_0^1 dV(s) = V(1)-V(0).
$$
It remains to study the term $I_{1,B,n}$. For this we obtain 
\begin{equation*}
\begin{split}
I_{1,B,n} &= \sum_{k=2}^n \frac{\vartheta(t_k,t_{k-1})}{2\sqrt{V(t_{k-1})}} \varphi\left(\frac{a}{\sqrt{V(t_{k-1})}}\right)\\
&= \sigma^2 n^{1-2H}\sum_{k=2}^n \frac{1}{2\sqrt{V(t_{k-1})}} \varphi\left(\frac{a}{\sqrt{V(t_{k-1})}}\right)(t_k-t_{k-1}) \\
&+ n^{1-2H}\sum_{k=2}^n \frac{g(t_k,t_{k-1})n^{2H}}{2\sqrt{V(t_{k-1})}} \varphi\left(\frac{a}{\sqrt{V(t_{k-1})}}\right)(t_k-t_{k-1}).
\end{split}
\end{equation*}
Here the first term satisfies, by Lemma \ref{lemma:leading-term-bound},
\begin{equation*}
\begin{split}
&\sigma^2 n^{1-2H}\sum_{k=2}^n \frac{1}{2\sqrt{V(t_{k-1})}} \varphi\left(\frac{a}{\sqrt{V(t_{k-1})}}\right)(t_k-t_{k-1}) \\
&= \sigma^2 n^{1-2H}\int_0^1 \frac{1}{2\sqrt{V(s)}} \varphi\left(\frac{a}{\sqrt{V(s)}}\right)ds + n^{1-2H}R'_{2,B,n},
\end{split}
\end{equation*}
where 
$$
n^{1-2H}R'_{2,B,n} \leq C\varphi(a)n^{H-1}\cdot n^{1-2H} = C\varphi(a)n^{-H}.
$$
The second term in turn satisfies, again by Lemma \ref{lemma:leading-term-bound},
\begin{equation*}
\begin{split}
&n^{1-2H}\sum_{k=2}^n \frac{g(t_k,t_{k-1})n^{2H}}{2\sqrt{V(t_{k-1})}} \varphi\left(\frac{a}{\sqrt{V(t_{k-1})}}\right)(t_k-t_{k-1}) \\
&\leq n^{1-2H}\max_{1\leq k\leq n} [g(t_k,t_{k-1})n^{2H}] \left[\int_0^1 \frac{1}{2\sqrt{V(s)}} \varphi\left(\frac{a}{\sqrt{V(s)}}\right)ds + R''_{2,B,n}\right]\\
&\leq n^{1-2H}\max_{1\leq k\leq n} [g(t_k,t_{k-1})n^{2H}] \varphi(a).
\end{split}
\end{equation*}
Collecting all the estimates completes the proof.
\end{proof}
\begin{remark}
\label{rem:upper-bound}
We note that by the above proof, we actually obtain 
$$
\E\left|\int_0^1 I_{X_s>a}dX_s - \sum_{k=1}^n I_{X_{t_{k-1}}>a}(X_{t_k}-X_{t_{k-1}})\right| \leq C\varphi(a)n^{1-2H}
$$
whenever we have only the upper bound $\E(X_t-X_s)^2 \leq C|t-s|^{2H}$ instead of \eqref{eq:variogram-assumption}. 
Indeed, the leading order term arises from $I_{1,B,n}$ with a constant given by 
$$
C(a) = \int_0^1 \frac{1}{\sqrt{V(s)}}\varphi\left(\frac{a}{\sqrt{V(s)}}\right)ds \leq \varphi(a) \int_0^1 \frac{1}{\sqrt{V(s)}}ds.
$$
\end{remark}
With the help of Proposition \ref{prop:main}, we are now ready to prove our main results.
\begin{proof}[Proof of Theorem \ref{thm:main}]
Using Lemma \ref{lem:convex-dif} and \eqref{eq:ito}, we have
\begin{equation*}
\begin{split}
&\Psi(X_1)-\Psi(X_0) - \sum_{k=1}^n \Psi'(X_{t_{k-1}})(X_{t_k}-X_{t_{k-1}}) \\
&= \sum_{k=1}^n \left[\Psi(X_{t_k})-\Psi(X_{t_{k-1}})-\Psi'(X_{t_{k-1}})(X_{t_k}-X_{t_{k-1}})\right] \\
&=\int_{\mathbb{R}} 2Z_n^+(a)\mu(da),
\end{split}
\end{equation*}
where 
$$
Z_n^+(a) = \sum_{k=1}^n \left[(X_{t_k}-a)^+-(X_{t_{k-1}}-a)^+-I_{X_{t_{k-1}}>a}(X_{t_k}-X_{t_{k-1}})\right].
$$
Taking expectation and using Proposition \ref{prop:main}, we get 
\begin{equation*}
\begin{split}
&\E \left|\Psi(X_1)-\Psi(X_0) - \sum_{k=1}^n \Psi'(X_{t_{k-1}})(X_{t_k}-X_{t_{k-1}})\right|\\
&= 2\int_{\mathbb{R}} \E Z_n^+(a)\mu(da)\\
&=\sigma^2\int_{\mathbb{R}}\int_0^1 \frac{1}{\sqrt{V(s)}} \varphi \left(\frac{a}{\sqrt{V(s)}}\right) ds\mu(da) \left(\frac{1}{n}\right)^{2H-1} + \int_{\mathbb{R}}R_n(a)\mu(da).
\end{split}
\end{equation*}
Here the remainder satisfies 
$$
R_n(a)\leq C\varphi(a)\max\{n^{-H},n^{1-2H}\max_{1\leq k\leq n} [g(t_k,t_{k-1})n^{2H}]\}
$$
that is integrable since $\int_{\mathbb{R}}\varphi(a)\mu(da) < \infty$ by assumption. Similarly, the leading order term is finite by the fact that 
$$
\int_0^1 \frac{1}{\sqrt{V(s)}}\varphi\left(\frac{a}{\sqrt{V(s)}}\right)ds \leq \varphi(a) \int_0^1 \frac{1}{\sqrt{V(s)}}ds \leq C\varphi(a).
$$
This yields the claim.
\end{proof}
\begin{proof}[Proof of Corollary \ref{cor:main}]
Let $A_K = \{\omega: \sup_{0\leq t\leq 1} |X_t| \leq K\}$. Since $f$ is locally of bounded variation, it follows that on the set $A_K$ we obtain 
$$
\int_0^1 \Psi'(X_s)dX_s - \sum_{k=1}^n \Psi'(X_{t_{k-1}})(X_{t_k}-X_{t_{k-1}}) 
= \int_{-K}^K Z_n^+(a)\mu(da).
$$
It follows that 
$$
\left|\int_0^1 \Psi'(X_s)dX_s - \sum_{k=1}^n \Psi'(X_{t_{k-1}})(X_{t_k}-X_{t_{k-1}})\right|\leq \int_{\mathbb{R}}Z_n^+(a)|\mu|(da).
$$
In view of Remark \ref{rem:upper-bound}, taking expectation yields the claim.
\end{proof}
\begin{proof}[Proof of Corollary \ref{cor:main2}]
The proof follows directly from the proof of Theorem \ref{thm:main} by considering lower and upper bounds separately, and hence we leave the details for an interested reader.
\end{proof}
\bibliographystyle{siam}
\bibliography{../../pipliateekki}

\end{document}